\documentclass[graybox]{svmult}

\usepackage{type1cm}        
\usepackage{makeidx}         
\usepackage{graphicx}        
\usepackage{multicol}        
\usepackage[bottom]{footmisc}

\usepackage{newtxtext}       %
\usepackage{newtxmath}       

\makeindex             

\begin{document}

\title{Restricted boundedness of translation operators	on variable Lebesgue spaces}
\author{Douadi Drihem}
\institute {M'sila University, Department of Mathematics, Laboratory	of Functional Analysis  and Geometry of Spaces \email{douadidr@yahoo.fr, douadi.drihem@univ-msila.dz}}


\maketitle

\abstract{In this paper, we investigate the inequality 
	\begin{equation*}
	\left\Vert f(\cdot +h)\right\Vert _{p\left( \cdot \right) }\leq A\left\Vert
	f\right\Vert _{p\left( \cdot \right) },\quad h\in \mathbb{R}^{n}, A>0
	\end{equation*}%
	under some suitable assumptions on the function $f$ and  the variable exponent $p$.}

\section{Introduction}

Function spaces with variable exponents have been intensively studied in the
recent years by a significant number of authors. The motivation for the
increasing interest in such spaces comes not only from theoretical purposes,
but also from applications to fluid dynamics \cite{Ru00}, image restoration 
\cite{CLR06} and PDEs with non-standard growth conditions. Some example of
these spaces can be mentioned such as: variable Lebesgue space, variable
Besov and Triebel-Lizorkin spaces. We only refer to the papers [1, 5-8]
and to the monograph \cite{DHHR11} for further details and references on
recent developments on this field.

The purpose of the present paper is to study the translation operators 
$\tau_{h}:f\mapsto f(\cdot +h)$, $h\in \mathbb{R}^{n}$ in the framework of
variable Lebesgue spaces $L^{p(\cdot )}(\mathbb{R}^{n})$. Their behavior is well known if $%
p$ is constant. In general $\tau _{h}$ maps $L^{p(\cdot )}(\mathbb{R}^{n})$
to $L^{p(\cdot )}(\mathbb{R}^{n})$ for any $\lambda >0 $ if and only
if $p$ is constant, see \cite{DHHR11}, Proposition 3.6.1. Allowing $p$ to
vary from point to point will raise extra difficulties which, in general,
are overcome by imposing some regularity assumptions on this exponent. By
these additional assumptions we ensure the boundedness of these operators on
variable Lebesgue spaces but with some appropriate assumptions. Before to state the main result, we fix some notation and recall some basics facts on variable Lebesgue spaces.
We denote by $B(x,r)$ the open ball in $\mathbb{R}^{n}$ with center $x$ and
radius $r$. By supp $f$ we denote the support of the function $f$ , i.e.,
the closure of its non-zero set.
By $\mathcal{S}(\mathbb{R}^{n})$ we denote the Schwartz space of all
complex-valued, infinitely differentiable and rapidly decreasing functions
on $\mathbb{R}^{n}$ and by $\mathcal{S}^{\prime }(\mathbb{R}^{n})$ the dual
space of all tempered distributions on $\mathbb{R}^{n}$. We define the
Fourier transform of a function $f\in \mathcal{S}(\mathbb{R}^{n})$\ by
\begin{equation*}
\mathcal{F}(f)(\xi ):=\left( 2\pi \right) ^{-n/2}\int_{\mathbb{R}	^{n}}e^{-ix\cdot \xi }f(x)dx.
\end{equation*}

The variable exponents that we consider are always measurable functions on $%
\mathbb{R}^{n}$ with range in $\left[ c,\infty \right) $ for some $c>0$. We
denote the set of such functions by $\mathcal{P}_{0}(\mathbb{R}^{n})$. The
subset of variable exponents with range $\left[ 1,\infty \right) $ is
denoted by $\mathcal{P}(\mathbb{R}^{n})$. We use the standard notation $%
p^{-}:=\underset{ x\in \mathbb{R}^{n}}{\text{ess-inf}}$ $p(x)$ and $p^{+}:=%
\underset{x\in \mathbb{R}^{n}}{\text{ess-sup }}p(x)$. Everywhere below we
shall consider bounded exponents.

The variable exponent Lebesgue space $L^{{p(\cdot )}}(\mathbb{R}^{n})$ is
the class of all measurable functions $f$ on ${\mathbb{R}^{n}}$ such that
the modular $\varrho _{{p(\cdot )}}(f):=\int_{\mathbb{R}^{n}}|f(x)|^{p(x)}%
\,dx$ is finite. This is a quasi-Banach function space equipped with the
quasi-norm $\Vert f\Vert _{p(\cdot )}:=\inf \big\{\mu >0:\varrho _{{p(\cdot )%
}}\big(\frac{f}{\mu }\big)\leq 1\big\}$. If $p(x):=p$ is constant, then $L^{{%
		\ p(\cdot )}}(\mathbb{R}^{n})=L^{p}(\mathbb{R}^{n})$ is the classical
Lebesgue space.

An useful property is that $\varrho _{{p(\cdot )}}(f)\leqslant 1$ if and
only if $\Vert f\Vert _{{p(\cdot )}}\leqslant 1$ (\textit{unit ball property}%
), which is clear for constant exponents since the relation between the norm
and the modular is obvious in that case. As is known, the following
inequalities hold%
\begin{equation}
\min\left (\varrho _{p(\cdot )}(f)^{1/p^{-}},\varrho _{p(\cdot
	)}(f)^{1/p^{+}}\right)\leq \left\Vert f\right\Vert _{p\left( \cdot \right) }\leq
\max\left (\varrho _{p(\cdot )}(f)^{1/p^{-}},\varrho _{p(\cdot )}(f)^{1/p^{+}}\right).
\label{sem-m-properties2}
\end{equation}

We say that a function $g\,:\,{\mathbb{R}^{n}}\rightarrow \mathbb{R}$ is 
\emph{locally $\log $-H\"{o}lder continuous}, if there exists a constant $%
c_{\log }>0$ such that $|g(x)-g(y)|\leq \frac{c_{\log }}{\ln (e+1/|x-y|)}$
for all $x,y\in \mathbb{R}^{n}$. If, for some $g_{\infty }\in \mathbb{R}$
and $c_{\log }>0$, there holds $|g(x)-g_{\infty }|\leq \frac{c_{\log }}{\ln
	(e+|x|)}$ for all $x\in \mathbb{R}^{n}$, then we say that $g$ satisfies the 
\textit{$\log $-H\"{o}lder decay condition} (at infinity). Note that every
function with log-decay condition is bounded.

The notation $\mathcal{P}^{\mathrm{log}}(\mathbb{R}^{n})$ is used for all
those exponents $p\in \mathcal{P}(\mathbb{R}^{n})$ which satisfy the local $%
\log $-H\"{o}lder continuity condition and the $\log $-H\"{o}lder decay
condition, where we consider $p_{\infty }:=\lim_{|x|\rightarrow \infty }p(x)$%
. The class $\mathcal{P}_{0}^{\mathrm{log}}(\mathbb{R}^{n})$ is defined
analogously.
By $c$ we denote generic positive constants, which may have different values
at different occurrences.
We refer to the recent monograph \cite{DHHR11} and the paper \cite{KR91} for further details, and historical remarks and references on variable exponent spaces.

In this paper we shall show the following result:
\begin{theorem}
	\label{key-estimate1}Let $p\in \mathcal{P}^{\mathrm{log}}(\mathbb{R}^{n})$
	\ with $1<p^{-}\leq p^{+}<\infty $, $h\in \mathbb{R}^{n}$\ and $k\in \mathbb{%
		\ N}$. \textit{Then for all }$f\in L^{p(\cdot )}(\mathbb{R}^{n})$\textit{\
		with }$\mathrm{supp}$\textit{\ }$\mathcal{F}f\subset \{\xi \in \mathbb{R}%
	^{n}:\left\vert \xi \right\vert \leq 2^{v+1}\},v\in \mathbb{N}_{0}$\textit{,
		we have} 
	\begin{equation*}
	\left\Vert \tau _{h}f\right\Vert _{p\left( \cdot \right) }\leq c\text{ }\exp
	\Big((2+2^{vnk}\left\vert h\right\vert ^{k})c_{\log }\left( 1/p\right)
	\Big)\left\Vert f\right\Vert _{p\left( \cdot \right) },
	\end{equation*}%
	where $c>0$ is independent of $h$, $v$ and $k$.
\end{theorem}
We mention that the boundedness of these operators in function spaces play
an important role in mathematical analysis. They appear in the localizations of Besov spaces  \cite{Naibo}, where  the author used the boundedness of these operators in Besov spaces which based on the Lebesgue spaces.
\section{Auxiliary results}
In this section we present some results which are useful for us. The next
lemma often allows us to deal with exponents which are smaller than $1$, see 
\cite[Lemma A.6]{DHR09}. Recall that $\eta _{v,m}\left( x\right) :=2^{nv}\left( 1+2^{v}\left\vert
x\right\vert \right) ^{-m}$, for any $x\in \mathbb{R}^{n}$, $v\in \mathbb{Z}$
and $m>0$. Note that $\eta _{v,m}\in L^{1}(\mathbb{R}^{n})$ when $m>n$ and
that $\left\Vert \eta _{v,m}\right\Vert _{1}=c_{m}$ is independent of $v$.

\begin{lemma}
	\label{r-trick}Let $r>0$, $v\in \mathbb{N}_{0}$ and $m>n$. Then there exists 
	$c=c(r,m,n)>0$ such that for all $g\in \mathcal{S}^{\prime }(\mathbb{R}^{n})$
	with $\mathrm{supp}$ $\mathcal{F}g\subset \{\xi \in \mathbb{R}
	^{n}:\left\vert \xi \right\vert \leq 2^{v+1}\}$, we have 
	\begin{equation*}
	\left\vert g(x)\right\vert \leq c(\eta _{v,m}\ast |g|^{r}(x))^{1/r},\quad
	x\in \mathbb{R}^{n}.
	\end{equation*}
\end{lemma}

We will make use of the following statement.

\begin{theorem}
	\label{DHHR-estimate}Let $p\in \mathcal{P}^{\mathrm{log}}(\mathbb{R}^{n})$, $%
	\theta :=2+\frac{\left\vert h\right\vert ^{k}}{|Q|^{k}}$, $k\in \mathbb{N}$, 
	$M:=\exp \frac{\theta c_{\log }\left( 1/p\right) }{p^{-}}$, if $|Q|<\min
	(\left\vert h\right\vert ,1)$ and $M:=1$, otherwise. Then for every $m>0$
	there exists   \ $\gamma =\exp \left( -4mc_{\mathrm{log}}\left( 1/p\right)
	\right) $ such that 
	$\Big(\frac{\gamma }{\left\vert Q\right\vert }\int_{Q}\left\vert \tau _{h}f(y)\right\vert dy\Big)^{p\left( x\right) }$ is	bounded by 
	\begin{equation*}
	\frac{cM}{\left\vert Q\right\vert }\int_{Q}\left\vert \tau _{h}f(y)\right\vert ^{p\left( y+h\right) }dy+cB\left(\left(
	e+\left\vert x\right\vert \right) ^{-m}+\frac{1}{\left\vert Q\right\vert }
	\int_{Q}\left( e+\left\vert y+h\right\vert \right) ^{-m}dy\right)
	\end{equation*}
 	for every cube $\mathrm{(}$or ball$\mathrm{)}$ $Q\subset \mathbb{R}^{n}$,
	all $x\in Q,$ $h\in \mathbb{R}^{n}$and all $f\in L^{p\left( \cdot \right) }( 
	\mathbb{R}^{n})+L^{\infty }(\mathbb{R}^{n})$\ with $\left\Vert f\right\Vert
	_{p\left( \cdot \right) }+\left\Vert f\right\Vert _{\infty }\leq 1$, where $B= \min %
	\big(1,\left\vert Q\right\vert ^{\frac{m}{\theta }}\big)$ and $c>0$ is independent of $h$, $k$ and $x$ .
\end{theorem}
\begin{proof}
	Our estimate use partially some decomposition techniques already used in  \cite[Theorem 4.2.4]{DHHR11}. Let $p\in 
	\mathcal{\ P}^{\mathrm{log}}(\mathbb{R}^{n})$ with $1\leq p^{-}\leq
	p^{+}<\infty $ and $p_{Q+h}^{-}=\underset{z\in Q}{\text{ess-inf}}$ $p(z+h)$.
	Define $q\in \mathcal{P}^{\mathrm{log}}(\mathbb{R}^{n}\times \mathbb{R}	^{n}\times \mathbb{\ R}^{n})$ by $ \frac{1}{q(x,y,h)}=\max\left(\frac{1}{p(x)}-\frac{1}{p(y+h)},0\right)$.  Then
	\begin{eqnarray*}
		&&\left( \frac{\gamma }{|Q|}\int_{Q}\left\vert f(y+h)\right\vert dy\right)
		^{p\left( x\right) } \\
		&\leq &\frac{M}{|Q|}\int_{Q}\left\vert f(y+h)\right\vert ^{p\left(
			y+h\right) }dy+\frac{1}{|Q|}\int_{Q}\gamma ^{q(x,y,h)}dy
	\end{eqnarray*}%
	for every cube $Q\subset \mathbb{R}^{n}$, all $x\in Q,$ $h\in \mathbb{R}^{n}$%
	and all $f\in L^{p\left( \cdot \right) }(\mathbb{R}^{n})+L^{\infty }(\mathbb{%
		\ R}^{n})$\ with $\left\Vert f\right\Vert _{p\left( \cdot \right)
	}+\left\Vert f\right\Vert _{\infty }\leq 1$. Indeed, we split $f\left(
	y+h\right) $ into three parts
	
	\begin{equation*}
	\begin{array}{ccc}
	f_{1}(y+h) & = & f(y+h)\chi _{\{y:\left\vert f(y+h)\right\vert >1\}}(y), \\ 
	f_{2}(y+h) & = & f(y+h)\chi _{\{y:\left\vert f(y+h)\right\vert \leq
		1,p(y+h)\leq p(x)\}}(y), \\ 
	f_{3}(y+h) & = & f(y+h)\chi _{\{y:\left\vert f(y+h)\right\vert \leq
		1,p(y+h)>p(x)\}}(y).%
	\end{array}%
	\end{equation*}%
	By convexity of $t\mapsto t^{p}$, 
	\begin{eqnarray*}
		\left( \frac{\gamma }{|Q|}\int_{Q}\left\vert f(y+h)\right\vert dy\right)
		^{p\left( x\right) } &\leq &3^{p^{+}-1}\sum\limits_{i=1}^{3}\left( \frac{
			\gamma }{|Q|}\int_{Q}\left\vert f_{i}(y+h)\right\vert dy\right) ^{p\left(
			x\right) } \\
		&=&3^{p^{+}-1}\left( I_{1}+I_{2}+I_{3}\right) .
	\end{eqnarray*}%
    
    	\textbf{Estimation of }$I_{1}$\textbf{.} We divide the estimation in three cases.
	
	\textbf{Case 1.} $p(x)\leq p_{Q+h}^{-}$. By Jensen's inequality,%
	\begin{equation*}
	I_{1}\leq \gamma ^{p\left( x\right) }\frac{1}{|Q|}\int_{Q}\left\vert
	f_{1}(y+h)\right\vert ^{p\left( x\right) }dy=I.
	\end{equation*}%
	Since $|f_{1}(y+h)|>1$, we have $\left\vert f_{1}(y+h)\right\vert
	^{p(x)}\leq \left\vert f_{1}(y+h)\right\vert ^{p_{Q+h}^{-}}\leq \left\vert
	f_{1}(y+h)\right\vert ^{p(y+h)}$ and thus
	 	\begin{equation*}I\leq \frac{1}{|Q|}	\int_{Q}\left\vert f(y+h)\right\vert ^{p(y+h)}dy.
	 		\end{equation*}
	  If $\left\Vert
f\right\Vert _{\infty }\leq 1$, then $f_{1}(y+h)=0$\ and $I=0$.
	
	\textbf{Case 2.} $p(x)>p_{Q+h}^{-}\geq p_{Q}^{-}$. Again Jensen's inequality implies that%
	\begin{eqnarray*}
		I_{1} &\leq &\left( \frac{\gamma }{|Q|}\int_{Q}\left\vert
		f_{1}(y+h)\right\vert ^{p_{Q}^{-}}dy\right) ^{\frac{p\left( x\right) }{
				p_{Q}^{-}}} \\
		&\leq &\left( \frac{\gamma }{|Q|}\int_{Q}\left\vert f(y+h)\right\vert
		^{p\left( y+h\right) }dy\right) ^{\frac{p\left( x\right) }{p_{Q}^{-}}
			-1}\left( \frac{\gamma }{|Q|}\int_{Q}\left\vert f(y+h)\right\vert ^{p\left(
			y+h\right) }dy\right) \\
		&\leq &c\frac{\gamma }{|Q|}\int_{Q}\left\vert f(y+h)\right\vert ^{p\left(
			y+h\right) }dy,
	\end{eqnarray*}%
	by the fact that $\int_{Q}\left\vert f(y+h)\right\vert ^{p\left( y+h\right)
	}dy\leq 1$ and $\left( \frac{1}{|Q|}\right) ^{\frac{p\left( x\right) }{
			p_{Q}^{-}}-1}\leq c$, which follow from\ $p\in \mathcal{P}^{\mathrm{log}}(%
	\mathbb{R}^{n})$, with $c>0$ independent of $x,h$ and $|Q|$.
	
	\textbf{Case 3.} $p(x)\geq p_{Q}^{-}>p_{Q+h}^{-}$. We have%
	\begin{eqnarray*}
		I_{1} &\leq &\left( \frac{\gamma }{|Q|}\int_{Q}\left\vert
		f_{1}(y+h)\right\vert ^{p_{Q+h}^{-}}dy\right) ^{\frac{p\left( x\right) }{
				p_{Q+h}^{-}}} \\
		&\leq &\left( \frac{\gamma }{|Q|}\int_{Q}\left\vert f(y+h)\right\vert
		^{p\left( y+h\right) }dy\right) \left( \frac{\gamma }{|Q|}\int_{Q}\left\vert
		f(y+h)\right\vert ^{p\left( y+h\right) }dy\right) ^{\frac{p\left( x\right) }{
				p_{Q+h}^{-}}-1} \\
		&\leq &\frac{1}{|Q|}\int_{Q}\left\vert f(y+h)\right\vert ^{p\left(
			y+h\right) }dy\left( \frac{1}{|Q|}\right) ^{\frac{p\left( x\right) }{
				p_{Q+h}^{-}}-1}.
	\end{eqnarray*}%
	If $|Q|\geq 1$, then the second term is bounded by $1$. Now we suppose that $%
	|Q|<1$. We use the local log-H\"{o}lder condition:%
	\begin{equation*}
	\left( \frac{1}{|Q|}\right) ^{\frac{p\left( x\right) }{p_{Q+h}^{-}}
		-1}=\left( \frac{1}{|Q|}\right) ^{\frac{p\left( x\right) -p_{Q}^{-}}{
			p_{Q+h}^{-}}}\left( \frac{1}{|Q|}\right) ^{\frac{p_{Q}^{-}-p_{Q+h}^{-}}{
			p_{Q+h}^{-}}}\leq c\left( \frac{1}{|Q|}\right) ^{\frac{p_{Q}^{-}-p_{Q+h}^{-} 
		}{p_{Q+h}^{-}}}.
	\end{equation*}%
	Let $p\left( x_{0}\right) =p_{Q}^{-}$ and $p\left( y_{0}+h\right)
	=p_{Q+h}^{-}$ with $x_{0},y_{0}\in Q$. Since $p\in \mathcal{P}^{\mathrm{log}%
	}(\mathbb{R}^{n})$, we have%
	\begin{equation*}
	\left( \frac{1}{|Q|}\right) ^{\frac{p_{Q}^{-}-p_{Q+h}^{-}}{p_{Q+h}^{-}}%
	}=\left( \frac{1}{|Q|}\right) ^{\frac{p\left( x_{0}\right) -p\left(
			y_{0}\right) }{p_{Q+h}^{-}}}\left( \frac{1}{|Q|}\right) ^{\frac{p\left(
			y_{0}\right) -p\left( y_{0}+h\right) }{p_{Q+h}^{-}}}\leq c\left( \frac{1}{|Q|%
	}\right) ^{\frac{p\left( y_{0}\right) -p\left( y_{0}+h\right) }{p_{Q+h}^{-}}%
	}.
	\end{equation*}%
	We see that%
	\begin{equation*}
	\left\vert p\left( y_{0}\right) -p\left( y_{0}+h\right) \right\vert \leq
	\sum\limits_{i=0}^{N-1}\Big|p\big(y_{0}+\frac{i}{N}h\big)-p\big(y_{0}+\frac{i+1}{N}h\big)%
	\Big|,
	\end{equation*}%
	where%
	\begin{equation*}
	N:=\left\{ 
	\begin{array}{ccc}
	\left[ \frac{\left\vert h\right\vert ^{k}}{|Q|^{k}}\right] +1, &  & 
	\left\vert h\right\vert >|Q|, \\ 
	1, & \text{otherwise.} & 
	\end{array}%
	\right.
	\end{equation*}%
	Therefore,%
	\begin{equation*}
	\left( \frac{1}{|Q|}\right) ^{\frac{p\left( y_{0}\right) -p\left(
			y_{0}+h\right) }{p_{Q+h}^{-}}}\leq \prod\limits_{i=0}^{N-1}\left( \frac{1}{
		|Q|}\right) ^{\frac{p(y_{0}+\frac{i}{N}h)-p(y_{0}+\frac{i+1}{N}h)}{
			p_{Q+h}^{-}}}\leq c\exp (Nc_{\log }\left( p\right) /p^{-}),
	\end{equation*}%
	where $c>0$ independent of $y_{0},h$, $N$ and $|Q|$, since%
	\begin{equation*}
	\Big|p\big(y_{0}+\frac{i}{N}h\big)-p\big(y_{0}+\frac{i+1}{N}h\big)\Big|\leq \frac{c_{\log
		}\left( 1/p\right) }{\log \Big(e+\frac{N}{\left\vert h\right\vert }\Big)}%
	\leq \frac{c_{\log }\left( 1/p\right) }{\log \Big(e+\frac{1}{\left\vert
			Q\right\vert }\Big)}
	\end{equation*}%
	if $\left\vert h\right\vert >|Q|$.
	
	\textbf{Estimation of}\textit{\ }$I_{2}$\textbf{.} By Jensen's inequality,%
	\begin{equation*}
	I_{2}\leq \gamma ^{p\left( x\right) }\frac{1}{|Q|}\int_{Q}\left\vert
	f_{2}(y+h)\right\vert ^{p\left( x\right) }dy=J.
	\end{equation*}%
	Since $|f_{2}(y+h)|\leq 1$ we have $\left\vert f_{2}(y+h)\right\vert
	^{p(x)}\leq \left\vert f_{2}(y+h)\right\vert ^{p(y+h)}$ and thus 
	\begin{equation*}
	J\leq \frac{1}{|Q|}\int_{Q}\left\vert f(y+h)\right\vert ^{p(y+h)}dy.
	\end{equation*}%
	
    	\textbf{Estimation of }$I_{3}$\textbf{.} Again by Jensen's inequality,%
	\begin{eqnarray*}
		&&\left( \frac{\gamma }{|Q|}\int_{Q}\left\vert f_{3}(y+h)\right\vert
		dy\right) ^{p\left( x\right) } \\
		&\leq &\frac{1}{|Q|}\int_{Q}\left( \left\vert \gamma f(y+h)\right\vert
		\right) ^{p\left( x\right) }\chi _{\{\left\vert f(y+h)\right\vert \leq
			1,p(y+h)>p(x)\}}(y)dy.
	\end{eqnarray*}%
	Now, Young's inequality give that the last term is bounded by%
	\begin{equation*}
	\frac{1}{|Q|}\int_{Q}\left( \left\vert f(y+h)\right\vert ^{p(y+h)}+\gamma
	^{q(x,y,h)}\right) dy.
	\end{equation*}%
	Observe that 
	\begin{equation*}
	\frac{1}{q(x,y,h)}=\max \Big(\frac{1}{p(x)}-\frac{1}{p(y+h)},0\Big)\leq \frac{1}{s(x)%
	}+\frac{1}{s(y+h)},
	\end{equation*}%
	where $\frac{1}{s(\cdot )}=\Big|\frac{1}{p(\cdot )}-\frac{1}{p_{\infty }}%
	\Big|$. We have 
	\begin{equation*}
	\gamma ^{q(x,y,h)}=\gamma ^{q(x,y,h)/2}\gamma ^{q(x,y,h)/2}\leq \gamma
	^{q(x,y,h)/2}\left( \gamma ^{s(x)/4}+\gamma ^{s(y+h)/4}\right) .
	\end{equation*}%
	We suppose that $\left\vert Q\right\vert <1$. Then%
	\begin{eqnarray*}
		\Big|\frac{1}{q(x,y,h)}\Big| &\leq &\Big|\frac{1}{p(x)}-\frac{1}{p(y)}\Big|+%
		\Big|\frac{1}{p(y)}-\frac{1}{p(y+h)}\Big| \\
		&\leq &\frac{c_{\log }\left( 1/p\right) }{-\log \left\vert Q\right\vert }%
		+\sum\limits_{i=0}^{N-1}\Big|\frac{1}{p(y+\frac{i}{N}h)}-\frac{1}{p(y+\frac{%
				i+1}{N}h)}\Big|.
	\end{eqnarray*}%
	Therefore,%
	\begin{eqnarray*}
		\Big|\frac{1}{q(x,y,h)}\Big| &\leq &\frac{c_{\log }\left( 1/p\right) }{-\log
			\left\vert Q\right\vert }+\sum\limits_{i=0}^{N-1}\frac{c_{\log }\left(
			1/p\right) }{\log \Big(e+\frac{N}{\left\vert h\right\vert }\Big)}\leq \frac{%
			c_{\log }\left( 1/p\right) }{-\log \left\vert Q\right\vert }\left( 1+N\right)
		\\
		&\leq &\frac{c_{\log }\left( 1/p\right) }{-\log \left\vert Q\right\vert }%
		\Big(2+\frac{\left\vert h\right\vert ^{k}}{|Q|^{k}}\Big).
	\end{eqnarray*}%
	Hence, $\gamma ^{q(x,y,h)/2}=\gamma ^{\frac{q(x,y,h)}{4}}\gamma ^{\frac{%
			q(x,y,h)}{4}}\leq \left\vert Q\right\vert ^{\frac{m}{2+\frac{\left\vert
				h\right\vert ^{k}}{|Q|^{k}}}}\gamma ^{\frac{q(x,y,h)}{4}}$. If $\left\vert
	Q\right\vert \geq 1$, then we use $\gamma ^{q(x,y,h)/2}\leq 1$ which follow
	from $\gamma <1$. Now by \cite[Proposition 4.1.8]{DHHR11}, we obtain the desired inequality. The proof is complete.
\end{proof}

\section{The proof of the main result}
In section we prove our result. Let  $p\in\mathcal{P}^{\mathrm{log}}(\mathbb{R}^{n})$ and define the translation operator by $(\tau _{h}f)(\cdot ):=f(\cdot +h)$.  We recall that the Hardy-Littlewood maximal operator $\mathcal{M}$ is defined on $L_{\mathrm{loc}}^{1}$ by
\begin{equation*}\mathcal{M}f(x):=\sup_{r>0}\frac{1}{|B(x,r)|}\int_{B(x,r)}|f(y)|dy,  \quad  x\in \mathbb{R}^{n}. \end{equation*}
Let 
\begin{equation*}M_{B(x,r)}f(x):=\frac{1}{\left\vert B(x,r)\right\vert }\int_{B(x,r)}\left%
\vert f(y)\right\vert dy,\quad r>0, x\in \mathbb{R}^{n}.\end{equation*}

\textbf{Proof of Theorem \protect\ref{key-estimate1}}.		Obviously, we assume that $\left\Vert f\right\Vert _{p\left( \cdot \right)
}\neq 0$. Since $L^{p(\cdot )}(\mathbb{R}^{n})\subset \mathcal{S}^{\prime }(%
\mathbb{R}^{n})$, Lemma \ref{r-trick} yields $\left\vert f\right\vert \leq
\eta _{v,N}\ast \left\vert f\right\vert $, for any\ $N>n,v\in \mathbb{N}_{0}$
. We write 
\begin{equation*}
\eta _{v,N}\ast |f|(x+h)=\int_{\mathbb{R}^{n}}\eta _{v,N}(x-y)\left\vert
f(y+h)\right\vert dy.
\end{equation*}%
We split the integral into two parts, one integral over the set $B(x,$ $%
2^{-v})$ and one over its complement. The first part is bounded by $M_{B(x%
	\text{, }2^{-v})}(\tau _{h}f)(x)$, and the second one is majorized by $%
c\sum_{i=0}^{\infty }2^{(n-N)i}M_{B(x\text{, }2^{1-v+i})}(\tau _{h}f)(x)$.
Consequently, where $N>n$,
\begin{equation*}
\left\Vert \tau _{h}f\right\Vert _{p\left( \cdot \right) }\leq
c\sum_{i=0}^{\infty }2^{(n-N)i}\big\|M_{B(\cdot \text{, }2^{1-v+i})}(\tau
_{h}f)\big\|_{p\left( \cdot \right) }.
\end{equation*}%
We will prove that 
\begin{equation*}
\Big\|\gamma \delta M_{B(\cdot \text{, }2^{1-v+i})}\Big(\frac{\tau _{h}f}{%
	\left\Vert f\right\Vert _{p\left( \cdot \right) }}\Big)\Big\|_{p\left( \cdot
	\right) }\leq c,\quad i,v\in \mathbb{N}_{0},
\end{equation*}%
with $c>0$ independent of $i,v$ and $h$, $\gamma =\exp\big ( -4mc_{\mathrm{%
		log}}\left( 1/p\right)\big)$ and $\delta =\exp \big(-(2+2^{vnk}\left\vert
h\right\vert ^{k})c_{\log }\left( 1/p\right) \big)$. Taking into account Theorem %
\ref{DHHR-estimate} we have, for any $i\in \mathbb{N}_{0},m>0$, 
\begin{equation}
\left(\gamma \text{ }\delta 2^{\left( v-i-1\right) n}\int_{B(x\text{, }%
	2^{1-v+i})}\frac{\big|\tau _{h}f(y)\big|}{\left\Vert f\right\Vert _{p\left(
		\cdot \right) }}dy\right)^{p\left( x\right) /p^{-}}  \label{key-exp}
\end{equation}%
we majorized it by, after a simple change of variable, 
\begin{equation*}
cM_{B(x+h\text{, }2^{1-v+i})}\Big(\left\vert g\right\vert ^{p\left( \cdot
	\right) /p^{-}}\Big)(x)+c\left( e+\left\vert x\right\vert \right)
^{-m}+cM_{B(x+h\text{, }2^{1-v+i})}\Big(\left( e+\left\vert \cdot \right\vert
\right) ^{-m}\Big)(x),
\end{equation*}%
with $g=\frac{f}{\left\Vert f\right\Vert _{p\left( \cdot \right) }}$ .
Therefore the expression $\mathrm{\eqref{key-exp}}$ is bounded by 
\begin{equation}
c\mathcal{M}\big(\left\vert g\right\vert ^{p\left( \cdot \right) /p^{-}}\big)%
(x+h)+c\left( e+\left\vert x\right\vert \right) ^{-m}+c\mathcal{M}%
(e+\left\vert \cdot \right\vert ^{-m})(x+h).  \label{key-exp2}
\end{equation}%
Obviously, 
\begin{equation*}
\varrho _{p(\cdot )}\Big(\gamma \delta M_{B(\cdot \text{, }2^{1-v+i})}\Big(%
\frac{\tau _{h}f}{\left\Vert f\right\Vert _{p\left( \cdot \right) }}\Big)%
\Big)=3^{p^{-}}\varrho _{p^{-}}\Big(\frac{1}{3}\Big(\gamma \delta M_{B(\cdot 
	\text{, }2^{1-v+i})}\Big(\frac{\tau _{h}f}{\left\Vert f\right\Vert _{p\left(
		\cdot \right) }}\Big)\Big)^{p\left( \cdot \right) /p^{-}}\Big).
\end{equation*}%
In view of $\mathrm{\eqref{key-exp2}}$, the last term can be estimated by 
\begin{equation*}
c\big\|\mathcal{M}(|g|^{p\left( \cdot \right) /p^{-}})(\cdot +h)%
\big\|_{p^{-}}^{p^{-}}+c\left\Vert \left( e+\left\vert \cdot
\right\vert \right) ^{-m}\right\Vert _{p^{-}}^{p^{-}}+c\left\Vert 
\mathcal{M}((e+\left\vert \cdot \right\vert )^{-m})(\cdot +h)\right\Vert
_{p^{-}}^{p^{-}}.
\end{equation*}%
First we see that $\left( e+\left\vert \cdot \right\vert \right) ^{-m}\in
L^{p^{-}}(\mathbb{R}^{n})$ for $m>\frac{n}{p^{-}}$. Secondly the classical
result on the continuity of $\mathcal{M}$ on $L^{p^{-}}(\mathbb{R}^{n})$
implies that 
\begin{equation*}
\big\|\mathcal{M}(|g|^{p\left( \cdot \right) /p^{-}})(\cdot +h)\big\|%
_{p^{-}}^{p^{-}}=\big\|\mathcal{M}\big(|g|^{p\left( \cdot \right) /p^{-}}%
\big)\big\|_{p^{-}}^{p^{-}}\leq c\text{ }\big\||g|^{p\left( \cdot \right)
	/p^{-}}\big\|_{p^{-}}^{p^{-}}=c\text{ }\varrho _{p(\cdot )}(g)\leq c,\text{ }
\end{equation*}%
and 
\begin{equation*}
\left\Vert \mathcal{M}((e+\left\vert \cdot \right\vert )^{-m})(\cdot
+h)\right\Vert _{p^{-}}^{p^{-}}=\left\Vert \mathcal{M}(e+\left\vert \cdot
\right\vert )^{-m}\right\Vert _{p^{-}}^{p^{-}}\leq c\text{ }\left\Vert
\left( e+\left\vert \cdot \right\vert \right) ^{-m}\right\Vert
_{p^{-}}^{p^{-}}\leq c,
\end{equation*}%
since $m>\frac{n}{p^{-}}$ (with $c>0$ independent of $h$). Hence 
\begin{equation*}
\varrho _{p(\cdot )}\Big(\gamma \delta M_{B(\cdot \text{, }2^{1-v+i})}\Big(%
\frac{\tau _{h}f}{\left\Vert f\right\Vert _{p\left( \cdot \right) }}\Big)%
\Big)\leq C\text{,}
\end{equation*}%
where $C>0$ independent of $i$ and $v$. Consequently, 
\begin{equation*}
\left\Vert \tau _{h}f\right\Vert _{p\left( \cdot \right) }\leq c\exp \big((2+2^{vnk}\left\vert h\right\vert ^{k})c_{\log }\left(
1/p\right)\big )\left\Vert f\right\Vert _{p\left( \cdot \right) }.
\end{equation*}%
The proof is complete.
\begin{remark}
	Using Lemma \ref{r-trick} we can extend Theorem  %
	\ref{key-estimate1} to the case where $p\in \mathcal{P}_{0}^{\mathrm{log}}(%
	\mathbb{R}^{n})$ with $0<p^{-}\leq p^{+}<\infty $.
\end{remark}

\textbf{Acknowledgements}. We thank the referee for carefully reading the paper and for making several useful suggestions and comments, which improved the exposition of the paper substantially.

\end{document}